\documentclass[12pt]{amsart}

\usepackage{amssymb}

\usepackage{amsmath}

\newtheorem{theorem}{Theorem}[section]
\newtheorem{lemma}[theorem]{Lemma}

\newtheorem{fact}[theorem]{Fact}

\newtheorem{corollary}[theorem]{Corollary}
\newtheorem{claim}[theorem]{Claim}
\theoremstyle{definition}
\newtheorem{definition}[theorem]{Definition}

\let \restr = \upharpoonright
\let \bs = \backslash
\let \into = \longrightarrow

\let \tld = \tilde

\let \sub = \subseteq
\let \elsub = \preccurlyeq

\let \av = \arrowvert
\let \ov = \overline
\let \a = \alpha
\let \b = \beta
\let \g = \gamma
\let \d = \delta
\let \e = \epsilon

\let \k = \kappa
\let \m = \mu
\let \n = \nu

\let \t = \theta

\let \D = \Delta
\let \T = \Theta
\let \s = \sigma
\let \x = \xi
\let \z= \zeta
\let \o = \omega

\let \al = \aleph
\let \la = \langle
\let \ra = \rangle
\let \mtcl = \mathcal
\let \mtbb = \mathbb
\let \it = \item

\title{Measuring club--sequences with large continuum}

\author[D. Asper\'o]{David Asper\'o}

\thanks{Mota was supported by the Austrian Science Fund FWF Project P22430. Both authors were also partially supported by Ministerio de
Educaci\'{o}n y Ciencia Project MTM2008--03389 (Spain) and by Generalitat de Catalunya Project 2009SGR--00187 (Cata\-lonia).}

\address{David Asper\'o, Institute of Discrete Mathematics and Geometry, TU Wien, Wiedner Hauptstrasse 8-10/104,
1040 Wien, Austria}

\email{david.aspero@tuwien.ac.at}

\author[M.A. Mota]{Miguel Angel Mota}

\address{Miguel Angel Mota, Kurt G\"odel Research Center for
  Mathematical Logic, W\"ahringer Stra\ss e 25, 1090 Wien,
  Austria}

\email{motagaytan@gmail.com}

\begin{document}

\maketitle

\date{}

\maketitle
\pagestyle{myheadings}
\markright{Measuring with large continuum}

\begin{abstract}
\emph{Measuring}, as defined by J. Moore, says that for every sequence $(C_\d)_{\d<\o_1}$ with each $C_\d$ being a closed subset of $\d$ there is a club $C\sub\o_1$ such that for every $\d\in C$,
a tail of $C\cap\d$ is either contained in or disjoint from $C_\d$. We answer a question of Moore by building a forcing extension satisfying measuring together with $2^{\al_0}>\al_2$. The construction works over any model of $\textsc{ZFC}$ and can be described as a finite support forcing iteration with systems of countable structures as side conditions and with symmetry constraints.
\end{abstract}

\section{Introduction}

One of the most frustrating problems faced by set theorists working with iterated proper forcing is the lack of techniques for producing models in which the continuum has size greater than the second uncountable cardinal.
In this paper we solve this problem in the specific case of measuring, a very strong negation of Club Guessing at $\o_1$ introduced by Justin Moore (see \cite{EISWORTH-MOORE}). This work is a natural continuation of our previous work in \cite{AM1} (see also \cite{AM2}), where we showed $2^{\al_0}>\al_2$ to be consistent together with a number of consequences of the Proper Forcing Axiom ($\textsc{PFA}$).

Our approach in \cite{AM1} consisted in building, starting from $\textsc{CH}$, a certain type of finite support forcing iteration of length $\k$ (in a general sense of `forcing iteration') using what one may describe as finite `symmetric systems' of countable elementary substructures of a fixed $H(\k)$\footnote{This $\k$ is exactly the value that $2^{\al_0}$ attains at the end of the construction.} as side conditions. These systems of structures were added at the first stage of the iteration.
Roughly speaking, the fact that the supports of the conditions in the iteration were finite ensured that the inductive proofs of the relevant facts -- mainly that the iteration has the $\al_2$--chain condition and that it is proper -- went through. The use of the sets of structures as side conditions was crucial in the proof of properness.\footnote{For more on the motivation for this type of construction see \cite{AM1} and \cite{AM2}.}

In the present paper we add a higher degree of `local' symmetry in the `single step' forcing notions involved and use it to build a model of measuring.

\begin{definition} (Moore, \cite{EISWORTH-MOORE})
\emph{Measuring} is the following statement: Let $\mtcl C=(C_\d)_{\d<\o_1}$ be such that each $C_\d$ is a closed subset of $\d$ (where $\d$ is endowed with the order topology). Then there is a club $C\sub\o_1$ which \emph{measures} $C_\d$ for every $\d\in C$. Specifically, this means that for every $\d\in C$ there is some $\a<\d$ with either $(C\cap\d)\bs\a\sub C_\d$ or $(C\bs\a)\cap C_\d=\emptyset$. We will also say that \emph{$C$ measures $\mtcl C$}.
\end{definition}

Measuring is of course equivalent to its restriction to club--sequences, where $(C_\d)_{\d<\o_1}$ is a club--sequence if every $C_\d$ is a club of $\d$.
Also, measuring clearly implies that for every ladder system $(C_\d)_{\d\in Lim\cap\o_1}$ there is a club $C\sub\o_1$ such that $C\cap C_\d$ is finite for all $\d\in C$,\footnote{$(C_\d)_{\d\in Lim\cap\o_1}$ is a ladder system if each $C_\d$ is a cofinal subset of $\d$ of order type $\o$.
The above statement is the negation of what is usually called \emph{Weak Club Guessing (for $\o_1$)}.} and that for every sequence $(f_\a)_{\a\in\o_1}$ , if each $f_\a$ is a continuous function from $\a$ into $\o$, then there is a club $C\sub\o_1$ such that $f_\d``\,C\neq\o$ for every $\d\in C$.\footnote{This is the negation of a statement Moore calls $\mho$ (mho) (see \cite{MOORE2}).}
Finally, it is easily seen to follow from $\textsc{PFA}$, and even from its bounded form $\textsc{BPFA}$.\footnote{The proof involves a forcing for adding a suitable club with side conditions. The `single--step forcing' in our construction is a variation of this type of forcing.}




Our main theorem is the following.

\begin{thm}\label{mainthm}
$(\textsc{CH}$) Let $\k$ be a cardinal such that $2^{<\k}=\k$ and $\k^{\al_1}=\k$. There is a proper poset $\mtcl P$ with the $\al_2$--chain condition such that the following statements hold after forcing with $\mtcl P$.

\begin{itemize}
\it[(1)] $2^{\al_0}=\k$

\it[(2)] Measuring
\end{itemize}
\end{thm}

The rest of the paper is devoted to proving Theorems \ref{mainthm}, and is organized as follows: Section \ref{the construction} starts with the central notion of symmetric system of structures. We then proceed, in Subsection \ref{The definition}, to the definition of a sequence $(\mtcl P_\a)_{\a\leq\k}$ of partial orders ($\mtcl P_\k$ will be shown to witness Theorem \ref{mainthm}).
In Section \ref{the proof} we give the basic analysis of our construction; in particular, we prove that it has the $\al_2$--chain condition, its properness, that $\mtcl P_\a$ embeds completely in $\mtcl P_\b$ if $\a<\b\leq\k$,
and that $\mtcl P_\k$ forces $2^{\al_0}=\k$ (Lemmas \ref{cc}, \ref{horribilis}, \ref{compll} and \ref{ch}). These results, together with the final lemma (Lemma \ref{measuring}), establish Theorem \ref{mainthm}.

For the most part our notation follows set--theoretic standards as set forth for example in \cite{JECH} and in \cite{KUNEN}.
If $N$ is a set whose intersection with $\omega_1$ is an ordinal, then $\delta_N$ will denote this intersection.
If $N$ is a set, $\mtbb P$ is a partial order and $G$ is a ($V$--generic) filter of $\mtbb P$, $N[G]$ will denote $\{\tau_G\,:\,\tau\in N,\,\tau\textrm{ a $\mtbb P$--name}\}$, where $\tau_G$ denotes the interpretation of $\tau$ by $G$. Also, $G$ is $N$--generic if $G\cap A\cap N\neq\emptyset$ for every maximal antichain $A$ of $\mtbb P$ belonging to $N$. A condition $p$ in $\mtbb P$ is $(N,\mtbb P)$ generic if $p$ forces that $\dot{G}$ is $N$--generic in the above sense. Note that we are \emph{not} assuming that $\mtbb P$ is in $N$ in any of these two sentences.

If $T$ is a predicate (i.e., a subset) of some $H(\t)$ and $\mtcl N=\la N, T\cap N\ra$ is a substructure of $\la H(\t), T\ra$, we also denote $\mtcl N$ by $\la N, T\ra$.\footnote{For example, $\la N, \in\ra$ will denote the structure $\la N, \in\cap\, N\times N\ra$.}
Sets $N$ will often be identified with the structure $\la N, \in\ra$. Recall that the elementary diagram of a structure $\la N, T\ra$ is the collection of sentences with parameters holding in $\la N, T\ra$.

Finally, we will consider the following natural notion of rank: Given two sets $\mtcl N$, $N$, we define the \emph{Cantor--Bendixson rank of $N$ with respect to $\mtcl N$},
$rank(\mtcl N, N)$, by specifying that $rank(\mtcl N, N)\geq 1$ if and only if for every $a\in N$ there is some $M\in\mtcl N\cap N$ such that $a\in M$ and, for each ordinal
$\m\geq 1$, that $rank(\mtcl N, N)>\m$ if and only if for every $a\in N$ there is some $M\in\mtcl N\cap N$ such that $a\in M$ and $rank(\mtcl N, M)\geq\m$.\footnote{Note that if $X$ is a set of ordinals and $\d$ is an ordinal, then $rank(X, \d)\geq 1$ if and only if $\d$ is a limit point of ordinals in $X$ and, for each ordinal
$\m\geq 1$, $rank(X, \d)>\m$ if and only if $\d$ is a limit of ordinals $\e$ with $rank(X, \e)\geq\m$.}

\section{Proving Theorem \ref{mainthm}: The construction}\label{the construction}

The proof of Theorem \ref{mainthm} will be given in a sequence of lemmas in this and the next section.

Let $\k\geq \o_2$ be a cardinal such that $2^{<\k}=\k$ and $\k^{\al_1}=\k$, and let $\Phi:\k\into H(\k)$ be a surjection such that for every $x$ in
$H(\kappa)$, $\Phi^{-1}(\{x\})$ is unbounded
in $\k$.

As anticipated in the introduction, $\mtcl P$ will be the final member $\mtcl P_{\k}$ of a certain sequence $\la\mtcl P_\a\,:\, \a \leq \k\ra$ of forcing notions.
Together with $\la\mtcl P_\a\,:\,\a\leq\k\ra$ we will also define a sequence $\la T^\a\,:\,\a<\k\ra$ of subsets of $H(\k)$\footnote{We will see the $T^\a$'s as truth predicates.} and a corresponding sequence $\la\mtcl M^\a\,:\,\a<\k\ra$ of clubs of $[H(\k)]^{\al_0}$.

\subsubsection{Symmetric systems of structures}

One central notion in our construction will be that of `symmetric system of structures'. We start by defining what we mean by this.

\begin{definition}\label{symm_system}
Let $\chi$ be an uncountable cardinal,  $\vec P$ a sequence $(P^\x)_{\x<\b}$ of subsets of $H(\chi)$ (for some ordinal $\b$), $\mtcl M$ a club of $[H(\chi)]^{\al_0}$, and $\mtcl N$ a finite set. We will say that \emph{$\mtcl N$ is a $\vec P$--symmetric $\mtcl M$--system} if the following conditions hold.

\begin{itemize}

\it[$(\a)$]  $\mtcl N \subseteq \mtcl M$.

\it[$(\b)$]  For all $N$, $N'\in\mtcl N$ and all $\x\in N\cap\b$, if $\d_N=\d_{N'}$ and $\x':=\Psi_{N, N'}(\x)<\b$, then there is a unique isomorphism $\Psi_{N, N'}$ between $\la N, P^{\x}\ra$ and $\la N', P^{\x'}\ra$.

\noindent Furthermore, we ask that $\Psi_{N, N'}$ be the identity on $\chi\cap N\cap N'$.

\it[$(\g)$] For all $N_0$, $N_1$ in $\mtcl N$, if $\d_{N_0}<\d_{N_1}$, then there is some $N_2\in\mtcl N$ such that $\d_{N_2}=\d_{N_1}$ and $N_0\in N_2$.

\it[$(\d)$] For all $N_0$, $N_1$, $N_2$ in $\mtcl N$, if $N_0\in N_1$ and $\d_{N_1}=\d_{N_2}$, then $\Psi_{N_1, N_2}(N_0)\in\mtcl N$.

\end{itemize}
\end{definition}

We may omit mentioning suitable parameters $\vec P$ and $\mtcl M$ when they are not relevant.
If $H(\chi)$ is understood or irrelevant we call $\mtcl N$ a \emph{symmetric system (of structures)}.

Throughout this paper, if $N$ and $N'$ are such that there is a unique isomorphism from $N$ into $N'$, then we denote this isomorphism by $\Psi_{N, N'}$.
The following facts are easy consequences from the above definition.

\begin{fact}\label{fact0} Let $\chi$, $\vec P = (P^\x)_{\x<\b}$ and $\mtcl M$ be as in Definition \ref{symm_system}. Let $\mtcl N$ be a $\vec P$--symmetric $\mtcl M$--system, let $N$, $N'\in\mtcl N$ be such that $\d_N=\d_{N'}$, and let $\x$, $\x'$ be ordinals in $\b$ such that $\x\in N$ and $\Psi_{N, N'}(\x)=\x'$. Then $\Psi_{N, N'}$ is an isomorphism between the structures $\la N, \in, P^\x, \mtcl N\cap N\ra$ and $\la N', \in, P^{\x'}, \mtcl N\cap N'\ra$.
\end{fact}

\begin{fact}\label{fact2}
Let  $\vec P=(P^\x)_{\x<\b}$ and $\mtcl M$ be as in Definition \ref{symm_system}. Let $\mtcl N_0=\{N^0_i\,:\,i<m\}$ and $\mtcl N_1=\{N^1_i\,:\,i<m\}$ be $\vec P$--symmetric $\mtcl M$--systems of structures.
Suppose that $(\bigcup\mtcl N_0)\cap(\bigcup\mtcl N_1)=X$, that there is an isomorphism $$\Psi:\la \bigcup_{i<m}N^0_1, \in, X\ra\into\la \bigcup_{i<m}N^1_i, \in, X\ra$$ fixing $X$, and that for all $\x\in\b\cap \bigcup_{i<m}N^0_i$, if $\x'=\Psi(\x)\in\b$, then $\la \bigcup_{i<m}N^0_i,\in, P^\x, X, N^0_i\ra_{i<m}$ and $\la \bigcup_{i<m}N^1_i,\in, P^{\x'}, X, N^1_{i}\ra_{i<m}$ are isomorphic structures. Then $\mtcl N_0\cup\mtcl N_1$ is a $\vec P$--symmetric $\mtcl M$--system of structures.
\end{fact}

\begin{proof}
The proof is a routine verification.
Let us show for example that if $i_0$, $i_1< m$ are such that $\d_{N^0_{i_0}}=\d_{N^1_{i_1}}$, then $\Psi_{N^0_{i_0}, N^1_{i_1}}$ fixes $Ord\cap N^0_{i_0}\cap N^1_{i_1}$: Let $\Psi$ be the isomorphism between $\la\bigcup_{i<m}N^0_i,\in, X, N^0_i\ra_{i<m}$ and $\la \bigcup_{i<m}N^1_i,\in, X, N^1_i\ra_{i<m}$. If $\g\in Ord\cap N^0_{i_0}\cap N^1_{i_1}$, then $\g\in X\cap N^0_{i_0}$, which implies that $\Psi(\g)=\g\in N^1_{i_0}\cap N^1_{i_1}$. But then $\g\in N^0_{i_0}\cap N^0_{i_1}$ as $\Psi$ is an isomorphism between the structures $\la\bigcup_{i<m}N^0_i,\in, X, N^0_i\ra_{i<m}$ and $\la \bigcup_{i<m}N^1_i,\in, X, N^1_i\ra_{i<m}$, which implies that $\Psi_{N^0_{i_0}, N^0_{i_1}}(\g)=\g$ and hence that $(\Psi\restr N^0_{i_1}\circ\Psi_{N^0_{i_0}, N^0_{i_1}})(\g) = \Psi_{N^0_{i_0}, N^1_{i_1}}(\g)=\g$.
\end{proof}

\subsection{The definition of $\la\mtcl P_\a\,:\,\a\leq\k\ra$}\label{The definition}
Now we proceed to the definition of $\la\mtcl P_\a\,:\,\a\leq \k\ra$ and $\la T^\a\,:\,\a<\k\ra$.

For every $\a< \k$, $\mtcl M^\a$ will be the set of $N\in [H(\k)]^{\al_0}$ such that $\la N, T^\a\ra \elsub\la H(\k), T^\a\ra$. We let $T^0\sub H(\k)$ code $\Phi$ together with the restriction of $\in$ to $H(\k)$. For every nonzero $\b < \k$, if $T^\a$ has been defined for all $\a<\b$, we let $\vec T^\b =(T^\a)_{\a<\b}$ and let
$T^\b\sub H(\k)$ code the elementary diagram of $\la H(\k), T^\a\ra_{\a<\b}$.\footnote{We will actually make use of $T^\b$ only for successor $\b$.} As we will see, each $\mtcl P_\a$ will be lightface definable in the structure $\la H(\k), T^{\a+1}\ra$. We let also $\mtcl M^\b_\ast = \{N\in [H(\k)]^{\al_0}\,:\,N\in\bigcap_{\x\in N\cap\b}\mtcl M^\x\}$.

We start with the definition of $\mtcl P_0$: A condition in $\mtcl P_0$ will be a pair $(\emptyset, \D)$, where

\begin{itemize}

\it[$(A1)$] $\D$ is a countable set of pairs of the form $(N, 0)$.

\it[$(A2)$] $dom(\D)$ is a $T^0$--symmetric $\mtcl M^0$--system of countable substructures of $H(\k)$.

\end{itemize}

Given $\mtcl P_0$--conditions $q_\e=(\emptyset, \D_\e)$ for $\e\in\{0, 1\}$, $q_1$ extends $q_0$ if and only if \begin{itemize} \it[$(B)$] $\D_0\sub\D_1$
\end{itemize}

In the definition of $\mtcl P_0$--condition we have used $0$ twice in a completely vacuous way. These (vacuous)
$0$'s are there to ensure that the (uniformly defined) operation of restricting a condition in a
(further) $\mtcl P_\b$ to an ordinal $\a<\b$ yields a condition in $\mtcl P_0$ when applied to any condition in any $\mtcl P_\b$ and to $\a=0$.

Given $\a\leq\k$ for which $\mtcl P_\a$ has been defined, let $\dot G_\a$ be a canonical $\mtcl P_\a$--name for the generic object and let $\mtcl N_{\dot G_\a}$ be a canonical $\mtcl P_\a$--name for the set $\bigcup\{\D_r^{-1}(\a)\,:\,r\in\dot G_\a\}$.\footnote{As we will see, conditions in $\mtcl P_\a$ will be pairs $r=(s, \D_r)$, with $\D_r$ a set of pairs $(N, \g)$, $\g$ an ordinal.}

Let $\b<\k$, $\b\neq 0$, and suppose that for all $\a<\b$

\begin{itemize}

\it[($\circ$)] we have defined $T^\a$ and $\mtcl P_\a$, and

\it[($\circ$)] $\mtcl P_\a\sub H(\k)$ is a partial order with the $\al_2$--chain condition consisting of pairs $r=(s, \D_r)$,
with $\D_r$ a set of pairs of the form $(N, \g)$, with $\g$ an ordinal.

\end{itemize}
\noindent The $\al_2$--c.c. of $\mtcl P_\a$ will follow from Lemma \ref{cc}.

We will use  a certain poset in $V^{\mtcl P_\a}$ for measuring a given club--sequence.

\subsubsection{A forcing notion for measuring a club--sequence}

Suppose $\dot{\mtcl C}$ is a $\mtcl P_\a$--name for a club--sequence $(C_\d)_{\d<\o_1}$. We define next a forcing $\T_{\dot{\mtcl C}}$, in $V^{\mtcl P_\a}$, for adding a club measuring $\dot{\mtcl C}$:

Conditions in $\T_{\dot{\mtcl C}}$ are triples $(f, b, \mtcl O)$ with the following properties.

\begin{itemize}

\it[(1)] $\mtcl O\sub\mtcl N_{\dot G_\a}$ is a $\vec T^{\a+2}$--symmetric $\mtcl M^{\a+2}_\ast$--system of structures.

\it[(2)] $f$ is a finite function that can be extended to a normal function $F:\o_1\into\o_1$ such that $F(\omega)=\omega$. Moreover, for every $\n\in dom(f)\setminus(\omega+1)$,

\begin{itemize}

\it[(2.1)] $f(\n)\in\{\d_N\,:\,N \in \mtcl O\}$, and
\it[(2.2)] $rank(\mtcl N_{\dot G_\a}\cap\mtcl M^{\a+2}_\ast, N)\geq\n$ for every $N\in \mtcl O$ such that $f(\n)=\d_N$. \end{itemize}

\it[(3)] $b$ is a function with domain included in $dom(f)\setminus(\omega+1)$. Moreover, the following holds for all $\n\in dom(b)$:

\begin{itemize}

\it[(3.1)] $b(\n)< \n$ and $b(\n)+1\in dom(f)$.

\it[(3.2)] If $\n_0\in dom(f)$ is such that $b(\n)< \n_0 < \n$, then $f(\n_0)\notin C_{f(\n)}$. Furthermore, if $\n_1\in dom(f)$ is such that $\n_0 + 1 < \n_1 <\n$, then
$[f(\n_0),\, f(\n_1)]\cap C_{f(\n)} = \emptyset$.

\it[(3.3)]  $rank(\{M \in \mtcl N_{\dot G_\a} \cap \mtcl M^{\a+2}_\ast\,:\,\d_M\notin C_{f(\n)}\}, N)\geq\n$ for every $N\in \mtcl O$ such that $f(\n)=\d_N$.

\end{itemize}

\end{itemize}

Given $\T_{\dot{\mtcl C}}$--conditions $(f_0, b_0, \mtcl O_0)$ and $(f_1, b_1, \mtcl O_1)$, $(f_1, b_1, \mtcl O_1)$ extends $(f_0, b_0, \mtcl O_0)$ in case $f_0\sub f_1$, $b_0\sub b_1$, and $\mtcl O_0\sub\mtcl O_1$.

The forcing $\T_{\dot{\mtcl C}}$ is meant to add a club $C$ measuring $\dot{\mtcl C}$. This club is the range of the union of all functions $f$ coming from conditions in the generic filter. The fact that this function is continuous and has domain $\o_1$ is ensured essentially by condition (2.2) in our definition. The function $b$ represents the commitment to avoid a certain member $C_\d$ of $\dot{\mtcl C}$ on a tail of $C\cap\d$ for every $\d$ in its domain. This, together with the conditions that must hold in case we make this commitment, is expressed in condition (3). By density, every $\n$ will eventually be in the domain of a relevant $b$ -- in other words, we will promise that $C$ avoids a tail of $f(\n)$ -- unless we cannot keep that promise. If we cannot keep that promise at a given $\n$, then a density argument using essentially condition (2) and the symmetry of $\mtcl O$ will show that a tail of $C\cap f(\n)$ will automatically go into $C_{f(\n)}$ (see the proof of Lemma \ref{measuring}).

The following lemma is immediate.

\begin{lemma}\label{l00}
Suppose $\dot{\mtcl C}$ is a $\mtcl P_\a$--name for a club--sequence.
If $(f, b, \mtcl O_0)$ and $(f, b, \mtcl O_1)$ are conditions in $\T_{\dot{\mtcl C}}$ and $\mtcl O_0\cup\mtcl O_1$ is a $\vec T^{\a+2}$--symmetric $\mtcl M^{\a+2}_\ast$--system,
then $(f, b, \mtcl O_0\cup\mtcl O_1)$ is a condition in $\T_{\dot{\mtcl C}}$ stronger than both $(f, b, \mtcl O_0)$ and $(f, b, \mtcl O_1)$.
\end{lemma}


\subsubsection{Resuming the construction}

We are now in a position to define $\mtcl P_\b$ in general for any $\b>0$, $\b\leq\k$ (assuming $\mtcl P_\a$ defined for all $\a<\b$).

If $\alpha < \kappa$ and $\mtcl P_\a$ is defined, we let $\Phi^\ast(\a)$ be a $\mtcl P_\a$--name for (say) the sequence $(\d)_{\d<\o_1}$ if $\Phi(\a)$ is not a $\mtcl P_\a$--name for a club--sequence, and let $\Phi^\ast(\a)$ be $\Phi(\a)$ if $\Phi(\a)$ is a $\mtcl P_\a$--name for a club--sequence.\footnote{It will follow from our definition that, for all $\a<\b$, a $\mtcl P_\a$--name is also a $\mtcl P_\b$--name (literally).}

Assume first that $\beta < \kappa$. Conditions in $\mtcl P_\b$ are pairs of the form $q= (p, \D)$ with the following properties.

\begin{itemize}

\it[$(C0)$] $p$ is a finite function such that $dom(p)\sub\b$ and $\D$ is a set of pairs $(N, \g)$ with $\g\leq\b$.

\it[$(C1)$]  $\D^{-1}(\b)$ is a $\vec{T}^{\b+1}$--symmetric $\mtcl M^{\b+1}_\ast$--system.


\it[$(C2)$] For every $\a<\b$, the restriction of $q$ to $\a$ is a condition in $\mtcl P_\a$. This restriction is defined
as $$q\av_\a:=(p\restr\a, \{(N, min\{\a, \g\})\,:\,(N, \g)\in \D\})$$

\it[$(C3)$] If $\a\in dom(p)$, then $p(\a)$ is of the form $\la f^{p, \a}, b^{p, \a}, \mtcl O^{p, \a}\ra$ and

\begin{itemize}

 \it[$(C3.1)$] $\D_{q\av_{\a+1}}^{-1}(\a+1)\cap\mtcl M^{\a+2}_\ast  \sub \mtcl O^{p, \a}\sub\D_{q\av_\a}^{-1}(\a)$,       \it[$(C3.2)$] $\mtcl P_\a\restr q\av_\a$ forces that $\la f^{p, \a}, b^{p, \a}, \mtcl O^{p, \a}\ra$ is a
$\Theta_{\Phi^\ast(\a)}$--condition, and \it[$(C3.3)$]
for all $N\in\D_{q\av_{\a+1}}^{-1}(\a+1)$, if $\a\in N$, then $\d_N\in dom(f^{p, \a})$ and $f^{p, \a}(\d_N)=\d_N$.
\end{itemize}

\end{itemize}

Given conditions $$q_\e =(p_\e, \,\D_\e)$$ (for $\e\in\{0, 1\}$) in $\mtcl P_\b$, we will say that $q_1 \leq_\b q_0$ if and only if the following holds.

\begin{itemize}

\it[$(D1)$] $q_1\av_\a \leq_\a q_0\av_\a$ for all $\a<\b$,

\it[$(D2)$] $dom(p_0)\sub dom(p_1)$

\it[$(D3)$] $f^{p_0, \a}\sub f^{p_1, \a}$, $b^{p_0, \a}\sub b^{p_1, \a}$ and $\mtcl O^{p_0, \a}\sub\mtcl O^{p_1, \a}$ for all $\a\in dom(p_0)$.\footnote{It follows of course that $q_1\av_\a$ forces that $(f^{p_1, \a}, b^{p_1, \a}, \mtcl O^{p_1, \a})$ extends $(f^{p_0, \a}, b^{p_0, \a}, \mtcl O^{p_0, \a})$ in $\T_{\Phi^\ast(\a)}$ whenever $\a\in dom(p_0)$.}

\it[$(D4)$] $\D_0^{-1}(\b)\sub \D_1^{-1}(\b)$.

\end{itemize}

\vspace{.3cm}
Note that if $\beta < \k$ is a nonzero limit ordinal and $q= (p, \D)$ satisfies  condition $(C0)$, then $q \in\mtcl P_\b$ iff $q$ satisfies $(C1)$ and $(C2)$.
Note also that for all $\b<\k$, $\mtcl P_\b$ is definable in $\la H(\k), T^{\b+1}\ra$.

We will use the following easy lemma.

\begin{lemma}\label{lemma0} For all $\b<\k$ and all $R\sub H(\k)$, if $M$ is such that $\la M, T^{\b+1}, R\ra\elsub\la H(\k), T^{\b+1}, R\ra$, then $\mtcl P_\b$ forces $\la M[\dot G_\b], \dot G_\b, R\ra\elsub\la H(\k)^{V[\dot G_\b]}, \dot G_\b, R\ra$.\end{lemma}

Finally we give the definition of the forcing $\mtcl P_\k$. Conditions in $\mtcl P_\k$ are pairs of the form $q= (p, \D)$ with the following properties.

\begin{itemize}

\it[$(E0)$] $p$ is a finite function such that $dom(p)\sub\kappa$ and $\D$ is a set of pairs $(N, \g)$ with $\g <\k$.

\it[$(E1)$] For every $\a<\k$, the restriction of $q$ to $\a$ is a condition in $\mtcl P_\a$. This restriction is defined
as $$q\av_\a:=(p\restr\a, \{(N, min\{\a, \g\})\,:\,(N, \g)\in \D\})$$

\end{itemize}

Given conditions $$q_\e =(p_\e, \,\D_\e)$$ (for $\e\in\{0, 1\}$) in $\mtcl P_\k$, we will say that $q_1 \leq_\k q_0$ if and only if the following holds.

\begin{itemize}

\it[$(F1)$] $q_1\av_\a \leq_\a q_0\av_\a$ for all $\a<\k$.

\end{itemize}

\section{Proving Theorem \ref{mainthm}: The actual proof}\label{the proof}

In this section we prove the main facts about the forcings $\mtcl P_\a$. Theorem \ref{mainthm} will follow immediately from them.

Our first lemma is immediate from the definitions.

\begin{lemma}\label{suborder} $\mtcl P_\k=\bigcup_{\b<\k}\mtcl P_\b$, and $\emptyset\neq\mtcl P_\a\sub \mtcl P_\b$ for all $\a\leq\b\leq\k$.\end{lemma}

Lemma \ref{compll} shows in particular that $\la\mtcl P_\a\,:\,\a\leq\k\ra$ is a forcing iteration in a broad sense.

\begin{lemma}\label{compll}
Let $\a \leq \b \leq \k$. If $q=(p, \D_{q})\in \mtcl P_\a$, $s=(r, \D_{s}) \in \mtcl P_\b$ and $q \leq_\a
s|_\a$, then $(p^{\smallfrown}(r\restr [\a,\, \b)), \D_{q}\cup \D_{s})$ is a condition in $\mtcl P_\b$ extending $s$.
Therefore, any maximal antichain in $\mtcl P_\a$ is a maximal antichain in $\mtcl P_\b$ and $\mtcl P_\a$ is a complete suborder of $\mtcl P_\b$.
\end{lemma}

\begin{proof}
It suffices to note that if the pair $(N, \g)$ is in $\D_q$, then the marker $\g$ (which bounds the influence of the side condition $N$ in clauses $(C1)$ and $(C3)$) is at most  $\alpha$.

\end{proof}

The following lemma shows that all forcings $\mtcl P_\b$ are $\al_2$--Knaster, and so in parti\-cu\-lar have the $\al_2$--chain condition.\footnote{A forcing $P$ is $\mu$--$Knaster$ if every
subset of $P$ of cardinality $\mu$ includes a subset of cardinality
$\mu$ of pairwise compatible conditions.} The proof uses standard $\D$--system arguments\footnote{This is the only place where we make use of $\textsc{CH}$.} (see Fact \ref{fact2} and Lemma \ref{l00}).

\begin{lemma}\label{cc}
For every $\b\leq\k$ and every set $\{(p_\x, \D_\x)\,:\,\x<\o_2\}$ of $\mtcl P_\b$--conditions there is $I\sub\o_2$ of size $\al_2$ such that for all $\x$, $\x'$ in $I$:
\begin{itemize}
\it[($\circ$)] if $\g \leq \beta$ and $\g < \k$, then $\D_{\x}^{-1}(\g) \cup \D_{\x'}^{-1}(\g)$ is a $\vec{T}^{\g+1}$--symmetric $\mtcl M^{\g+1}_\ast$--system of structures,
\it[($\circ$)] if $\a\in dom(p_\x)\cap dom(p_{\x'})$, $(f^{p_\x, \a}, b^{p_\x, \a})=(f^{p_{\x'}, \a}, b^{p_{\x'}, \a})$ and $\mtcl O^{p_\x, \a}\cup\mtcl O^{p_{\x'}, \a}$ is a $\vec{T}^{\a+2}$--symmetric $\mtcl M^{\a+2}_\ast$--system of structures, and

\it[($\circ$)] letting $p^\ast(\a)=(f^{p_\x, \a}\cup f^{p_{\x'}, \a}, b^{p_\x, \a}\cup b^{p_{\x'}, \a}, \mtcl O^{p_\x, \a}\cup\mtcl O^{p_{\x'}, \a})$ for all $\a\in dom(p_\x\cup p_{\x'})$, $(p^\ast, \D_\x\cup\D_{\x'})$ is a $\mtcl P_\b$--condition extending both $(p_\x, \D_\x)$ and $(p_{\x'}, \D_{\x'})$.\footnote{If $\a$ is not in the domain of $p_\x$, then by $(f^{p_\x, \a}, b^{p_\x, \a}, \mtcl O^{p_\x, \a})$ we will mean $(\emptyset,\emptyset,\emptyset)$ and similarly for $p_{\x'}$.}
\end{itemize}

In particular, $\mtcl P_\b$ is $\al_2$--Knaster.
\end{lemma}

\begin{corollary}\label{susto}
If $\dot{\mtcl C}$ is a $\mtcl P_\a$--name for a club--sequence, then $\mtcl P_\a$ forces $\T_{\dot{\mtcl C}}$ to have the $(\al_2)^V$--chain condition.\footnote{By lemma \ref{horribilis} we will see that $(\al_2)^V=(\al_2)^{V^{\mtcl P_\a}}$.} Hence, every maximal antichain of $\T_{\dot{\mtcl C}}$ is forced to be a member of $H(\k)$.\end{corollary}

\begin{proof}
Suppose $\dot{A}$ is a $\mtcl P_\a$--name for a maximal antichain of $\T_{\dot{\mtcl C}}$ of size $(\al_2)^V$. For each $\zeta \in \o_2$, let $p_\zeta$ be a $\mtcl P_\a$--condition forcing $\dot{a}_{\zeta}=\check{b}_{\zeta}$ for some $b_\zeta$ in $V$, where $\dot{a}_{\zeta}$ denotes the  $\zeta$-th element of $\dot{A}$. By the above lemma, we may assume that all the $p_\zeta$'s are pairwise $\mtcl P_\a$--compatible. Using Fact \ref{fact2} and Lemma \ref{l00}, we can also assume that any common extension of $p_\zeta$ and $p_{\zeta'}$ forces that $\check{b}_{\zeta}$ and $\check{b}_{\zeta'}$ are $\T_{\dot{\mtcl C}}$--compatible, but that is a contradiction.
\end{proof}

Strictly speaking, the following lemma will not be needed in the rest of the paper. However, its proof is a convenient warm-up for the proof of conclusion $(2)_\b$ of Lemma \ref{horribilis}.

\begin{lemma}\label{preproper}
Let $\b<\k$, and suppose $q=(p,\D_q)$ is $(M,\,\mtcl P_\b)$--generic whenever $q\in\mtcl P_\b$ and $M\in\D^{-1}_q(\b) \cap\mtcl M^{\b+1}$.\footnote{We will see in Lemma \ref{horribilis} that this hypothesis is true.} If $\dot{\mtcl C}$ is a $\mtcl P_\b$--name for a club--sequence $(C_\d)_{\d<\o_1}$, then $\mtcl P_\b$ forces that $\s_0=(f_0, b_0, \mtcl O_0)$ is $(N[\dot G_\b],\,\T_{\dot{\mtcl C}})$--generic whenever $\s_0\in\T_{\dot{\mtcl C}}$, $N\in\mtcl O_0\cap\mtcl M^{\b+2}$, $\d_N\in dom(f_0)$, and $f_0(\d_N)=\d_N$.
\end{lemma}

\begin{proof}
Let us work in $V^{\mtcl P_\b}$. Let $A\in N[\dot G_\b]$ be a maximal antichain of $\T_{\dot{\mtcl C}}$. By extending $\s_0$ if necessary we may assume that $\s_0$ extends a condition $\s^\dag$ in $A$. We want to show of course that $\s^\dag$ is in $N[\dot G_\b]$, and for this it will suffice to find a member of $A\cap N[\dot G_\b]$ compatible with $\s_0$.

We may assume that $\d_N\in dom(b)$, as otherwise the proof is slightly simpler. Let $\m= max(range(f_0\restr\d_N))+1$,\footnote{Note that $range(f_0\restr\d_N) \neq \emptyset$ by condition (3.1) in the definition of $\T_{\dot{\mtcl C}}$.} let $\tld A\in N[\dot G_\b]$ be the (partially defined) function sending each $\s\in A$ to the first $\T_{\dot{\mtcl C}}$--condition $(f', b', \mtcl O')$ extending $\s$ (in some canonical well--ordering given by $\Phi$) and such that $f_0\restr\d_N \sub f'$, $range(f')\cap\m=range(f_0)\cap\m$, $b_0\restr\d_N\sub b'$ and $\mtcl O_0\cap N\sub\mtcl O'$ (whenever this is possible),\footnote{$\tld A$ is in $N[\dot G_\b]$ since $(f_0\restr\d_N, b_0\restr\d_N, \mtcl O_0)\in N$ and $\tld A$ is definable in the structure $\la H(\k)^{V[\dot G_\b]}, \dot G_\b, T^{\b+2}\ra$ from $(f_0\restr\d_N, b_0\restr\d_N, \mtcl O_0)$ (by Lemma \ref{lemma0}).} and let $M\in N\cap \mtcl N_{\dot G_\b}\cap\mtcl M^{\b+2}_\ast$ be such that $\tld A\in M[\dot G_\b]$, $\b+1\in M$, and $\d_M\notin C_{\d_N}$. Such an $M$ exists by condition (3.3) in the definition of $\T_{\dot{\mtcl C}}$. Let $\eta<\d_M$ be such that $[\eta,\,\d_M]\cap C_{\d_N}=\emptyset$ (this $\eta$ exists by openness of $\d_N\setminus C_{\d_N}$). Now, in $M[\dot G_\b]$ there is $\s_\ast=(f_\ast, b_\ast, \mtcl O_\ast)$
such that

\begin{itemize}

\it[(a)] $\s_\ast\in range(\tld A)$, and

\it[(b)] $min(range(f_\ast)\bs\m)>\eta$.

\end{itemize}

Note that $max(range(f_\ast))<\d_M$ since $\d_{M[\dot G_\b]}= \d_M$ by our assumption on $\mtcl N_{\dot G_\b}\cap\mtcl M^{\b+1}$. Let $$\mtcl O' = \mtcl O_0 \cup \{\Psi_{N, N'}(M)\,:\,M\in\mtcl O_\ast,\,N'\in \mtcl O_0,\,\d_{N'}=\d_N\}$$ Now it is easy to check that $(f_\ast\cup f_0, b_\ast\cup b_0, \mtcl O')$ is a common extension of $\s_\ast$ and $\s_0$ in $\Theta_{\dot{\mtcl C}}$ (this uses condition $(C1)$ in the definition of
$\mtcl P_\b$ for the verification of conditions $(1)$, $(2.2)$ and $(3.3)$ in the definition of $\Theta_{\dot{\mtcl C}}$). Letting $\ov\s\in A\cap N[\dot G_\b]$ be the unique $\s\in A$ such that $\tld A(\s)=\s_\ast$,\footnote{Note that $\tld A$ is one--to--one.}
it follows that $\ov\s=\s^\dag$.
\end{proof}

The following lemma can be proved easily by induction on $\a$. It will be used in the proof of Lemma \ref{horribilis}.

\begin{lemma}\label{self-deciding} For all $\a<\k$, $v\in\mtcl P_\a$, and $\d<\o_1$ there is $w=(\ov w, \D_w)\in\mtcl P_\a$ extending $v$ together with $\eta<\d$ such that for all $M\in N$ and all $\x\in dom(\ov w)\cap N$, if $M$ and $N$ are both in $\D_{w\av_{\x+1}}^{-1}(\x+1)$,  $\d_N = \d \in dom(b^{\ov w, \x})$ and $b^{\ov w, \x}(\d_N)<\d_M = f^{\ov w, \x}(\d_M)$, then $w\av_\x\Vdash_{\mtcl P_\x} \dot{C}^\x_{\d_N}\cap[\eta,\,\d_M]=\emptyset$.\end{lemma}

The properness of all $\mtcl P_\b$ ($\beta < \kappa$) is an immediate consequence of the following lemma.

\begin{lemma}\label{horribilis}

Suppose $\b < \kappa$ and $N \in \mtcl M^{\b+1}$. Then the following conditions hold.

\begin{itemize}

\it[$(1)_\b$] For every $q\in N \cap \mtcl P_\b$ there is $q'\leq_\b q$ such that $N \in \D^{-1}_{q'}(\b)$.

\it[$(2)_\b$] If $q\in\mtcl P_\b$ and $N \in \D^{-1}_{q}(\b)$, then $q$ is $(N,\, \mtcl P_\b)$--generic.

\end{itemize}

\end{lemma}

\begin{proof}
The proof of  $(2)_\b$ will be the same in all cases, and the proof of $(1)_\b$ will be by induction on $\b$. The proof of $(1)_0$ is trivial: It suffices to set $q'= q\cup\{(N, 0)\}$.


The proof of $(1)_\b$ when $\b=\a+1$ is as follows. Let $q= (p, \D_q)$. By $(1)_\a$ we may assume that there is a condition $t=(u, \D_t) \in \mtcl P_\a$ extending $q\av_\a$ and such that $N \in \D^{-1}_{t}(\a)$. This condition $t$ clearly forces (in $\mtcl P_\a$) that
$N \in \mtcl N_{\dot G_\a}$. So, $t$ forces that for every $x \in N$, there is $M\in \mtcl N_{\dot G_\a} \cap \mtcl M^{\a+1}$ such that $x \in M$.

Let us work in $V^{\mtcl P_\a\restr t}$. Since, by Lemma \ref{lemma0}, $\la N[\dot G_\a], \dot G_\a, T^{\a+2}, H(\k)^V\ra$ is an elementary substructure of $\la H(\k)[\dot G_\a], \dot G_\a, T^{\a+2}, H(\k)^V\ra$, there exists an $M$ as above in $N[\dot G_\a] \cap V$ (where $V$ denotes the ground model). We can also assume that $M \in N$, since $N[\dot G_\a]\cap V = N$ (which follows from $(2)_\a$ applied to $N$ and $t$). This shows that $t$ forces $rank(\mtcl N_{\dot G_\a} \cap \mtcl M^{\a+1}, N) \geq 1$. In fact, a similar argument shows that $t$ forces $rank(\mtcl N_{\dot G_\a} \cap \mtcl M^{\a+1}, N) > \mu$ for every $\mu < \delta_N$. In view of these considerations, it suffices to define $q'$ as the condition $(u', \D_q \cup \D_{t} \cup \{(N,\b)\})$, where $u'$ extends $u$ and sends the ordinal $\alpha$ to the triple $(f^{p, \a}\cup\{\la\d_N, \d_N\ra\}, b^{p, \a}, \mtcl O^{p, \a}\cup\{N\})$.

The proof of $(1)_\b$ when $\b$ is a nonzero limit ordinal is trivial using $(1)_\a$ for all $\a < \b$, together with the fact if $q=(p,\D) \in \mtcl P_\b$, then the domain of $p$ is bounded in $\b$.

Now let us proceed to the proof of $(2)_\b$ for general $\b$. Let $A\sub\mtcl P_\b$ be a maximal antichain in $N$, and suppose $q=(p,\D_q)$ extends a condition in $A$.
We want to see that $q$ is compatible with a condition in $A\cap N$.











Let $\x_0$ be the maximum of the set $X$ of $\x\in dom(p)$ such that $\x\in N'$ for some pair $(N', \g)\in\D_q$ with $\x<\g$ and $\d_{N'}=\d_N$. Let $(N', \g)\in \D_q$ witness $\x_0\in X$. By extending $q$ further if necessary, we may assume that there is some $M'\in N'$ such that $(M', \x_0)\in \D_q$ and such that $\Psi_{N, N'}(x)\in M'$ for all relevant $x\in N$.
Let $M=\Psi_{N', N}(M')$. If $\d_N \in dom(b^{p, \x_0})$, we may assume that there is $\eta<\d_M$ such that $q\av_{\x_0}$ forces $\dot C^{\x_0}_{\d_N}\cap[\eta,\,\d_M]=\emptyset$.
By Lemma \ref{self-deciding} we may further assume that $q\av_\x$ forces $\dot C^\x_{\d_N}\cap[\eta,\,\d_M]=\emptyset$ whenever $\x\in X$, $\d_N= \d_{\ov N}\in dom(b^{p, \x})$, $\ov M\in \ov N$, $\d_{\ov M}=\d_M$, $\ov M$, $\ov N$ are both in $\D_{q\av_{\x+1}}^{-1}(q\av_{\x+1})$, and $b^{p, \x}(\d_N)<\d_M = f^{p, \x}(\d_M)$.

\begin{claim} \label{toronto}
Let $\x \in X$.

\begin{itemize}

\it[(a)] If $\x = \x_0$ and $\d_N \in dom(b^{p, \x_0})$, then $q\av_{\x_0}$ forces $\dot C^{\x_0}_{\d_N}\cap[\eta,\,\d_M]=\emptyset$.
\it[(b)] If $\x \neq \x_0$, $(\ov N, \ov \g)\in\D_q$, $\delta_N= \delta_{\ov N}$, $\x \in \ov N \cap \ov \g \cap \Psi_{N,\ov N }(M)$ and $\d_N \in dom(b^{p, \x})$, then $(\Psi_{N,\ov N }(M), \x+1) \in \D_{q|_{\x+1}}$ and $\delta_M$ is a fixed point of $f^{p, \x}$. In particular, these hypothesis imply that $q\av_\x$ forces $\dot C^\x_{\d_N}\cap[\eta,\,\d_M]=\emptyset$.
\end{itemize}

\end{claim}

\begin{proof}
In order to prove $(b)$, it is enough to note that $\Psi_{N,\ov N }(M) =\Psi_{N',\ov N }(M')$ and to apply clauses $(C1)$ and $(C3.3)$ to condition $q\av_{\x+1}$.
\end{proof}

Let $\{\d_0,\ldots\d_{l-1}\} =\{\d_{N'}\,:\,N'\in\D_{q\av_0}^{-1}(0)\}\cap\d_M$. By correctness of $M$ and since $M$ contains all relevant objects, there is a condition $t = (\ov t, \D_t)\in M$ satisfying the following properties (in $V$).

\begin{itemize}

\it[(1)] $t\in A$.

\it[(2)] For all $W$ in $\D_{q\av_0}^{-1}(0)\cap M$ and for all $\z \in \beta+1$, if $\z \in W$ and $W\in \D_{q\av_\z}^{-1}(\z)$, then $W \in \D_{t\av_\z}^{-1}(\z)$.\footnote{Note that for such a $W$, the set of those $\zeta$ in $M \cap(\beta+1)$ such that $\zeta \in W$ and $W\in \D_{q\av_\z}^{-1}(\z)$ can be correctly computed (in M) by means of a formula using as parameters the structure $W$ and the minimum ordinal in $M$ which is at least the maximum of those $\varsigma$ such that $(W, \varsigma)\in \D_q$.}

\it[(3)] For all $\z\in dom(\ov t)$, if $\z\in dom(p)$, then

\begin{itemize}

\it[(3.1)] $(f^{\ov t, \z}, b^{\ov t, \z}, \mtcl O^{\ov t, \z})$ and $(f^{p, \z}, b^{p, \z}, \mtcl O^{p, \z})$ are forced by $q\av_\z$ to be compatible as $\Theta_{\Phi^\ast(\z)}$--conditions, and

\it[(3.2)] the least point in $dom(f^{\ov t, \z})$ above $dom(f^{p, \z})$ is above $\eta$.
\end{itemize}

\it[(4)] For all $i<l$, for all $\x \in dom(p)$ and for all $(N', \g)\in\D_q$ with $\x<\g$ and $\d_{N'}=\d_N$, if there is no $W$ such that $\x \in W$, $W\in \D_{q\av_{\x+1}}^{-1}(\x+1)$ and $\delta_W=\delta_i$, then letting $\eta = \Psi_{N', N}(\x)$ and $\rho = min((OR\cap M)\setminus \x)$, there is no $W$ such that $\eta \in W$, $W\in \D_{t\av_{\rho}}^{-1}(\rho)$ and $\delta_W=\delta_i$.\footnote{By clause $(C1)$ applied to condition $q\av_{\x+1}$, the existence of a $W$ such that $\Psi_{N', N}(\x) \in W\in \D_{t\av_{\rho}}^{-1}(\rho)$ would imply that $\x = \Psi_{N, N'}(\Psi_{N', N}(\x)) \in \Psi_{N, N'}(W) \in \D_{q\av_{\x+1}}^{-1}(\x+1)$.}


\it[(5)] If $(M', \g)\in\D_t$ and $\d_{M'}\notin\{\d_0,\ldots \d_{l-1}\}$, then $\d_{M'}>\eta$.

\end{itemize}

Using Claim \ref{toronto}, it is easy to check that one can amalgamate $q$ and $t$ into a condition $q^\ast$ with $\D_{q^\ast}$ being the union of $\D_q$ with the set of all pairs $(\Psi_{N, N'}(W), min\{\rho,\, \g\})$ such that $(W, \rho)\in \D_t$, $(N', \g) \in\D_q$, and $\d_{N'}=\d_N$. The closure of $\D_{q^\ast}$ under isomorphisms does not interfere with the elements of $X$ because of condition $(4)$.

\end{proof}

\begin{corollary}\label{proper}
For every $\b\leq \k$, $\mtcl P_\b$ is proper.
\end{corollary}

\begin{proof} For $\b<\k$ the conclusion follows immediately from Lemma \ref{horribilis}. The remaining case follows from the corresponding conclusions for $\b<\k$ together with the $\al_2$--c.c. of $\mtcl P_\k$ and $cf(\k)\geq\o_2$.
\end{proof}

For every $\b<\k$ let $\dot F_\b$ and $\dot B_\b$ be $\mtcl P_\k$--names for, respectively, the union of all functions $f$ for which there is a condition $q=(p, \D)\in\dot G_\k$ such that $p(\b)=(f, b, \mtcl O)$ for some $b$ and $\mtcl O$, and the union of all $b$ for which there is a condition $q=(p, \D)\in\dot G_\k$ such that $p(\b)=(f, b, \mtcl O)$ for some $f$ and $\mtcl O$.

\begin{lemma}\label{ch}
$\mtcl P_\k$ forces $2^{\aleph_0}=\kappa$.

\end{lemma}
\begin{proof}
In order to prove that $\mtcl P_\k$ forces $2^{\aleph_0} \geq \kappa$, it suffices to note that if $\beta < \kappa$, then the restriction of $\dot F_\b$ to $\omega$ is forced to be a Cohen real (recall that $\dot F$ is a name for a normal function having $\omega$ as a fixed point).
The other inequality follows from counting nice names for subsets of $\omega$ using the fact that $\k^{\o_1}=\k$ together with Lemma \ref{cc}.
\end{proof}

\begin{lemma}\label{measuring}
For all $\b<\k$, $\mtcl P_\k$ forces that $range(\dot F_\b)$ is a club of $\o_1$ measuring $\Phi^\ast(\b)$.
\end{lemma}

\begin{proof}
Let $\dot C_\d$ be, for each $\d$, a $\mtcl P_\b$--name for the $\d$-th member of $\Phi^\ast(\b)$. We want to show that the following conditions hold in $V^{\mtcl P_\k}$:

\begin{itemize}

\it[(A)] $\dot F_\b$ is a normal function with domain $\o_1$.

\it[(B)] For each $\n<\o_1$,

\begin{itemize}

\it[(B1)] if $\n\in dom(\dot B_\b)$, then $range(\dot F_\b\restr (\dot B_\b(\n),\,\n))$ is disjoint from $\dot C_{\dot F(\n)}$, and

\it[(B2)] if $\n\notin dom(\dot B_\b)$, then a tail of $range(\dot F_\b\restr\n)$ is included in $\dot C_{f(\n)}$.
\end{itemize}\end{itemize}

Showing (A) is easy, so here we will only show (B). Note that for every $q=(p, \D)\in\mtcl P_\k$, if $p(\b)=(f, b, \mtcl O)$ and $\n\in dom(f)$, then there is some $q'=(p', \D')$ extending $q$ such that, letting $p'(\b)=(f', b', \mtcl O')$, either

\begin{itemize}
\it[(a)] $\n\in dom(b')$, or

\it[(b)] $q'$ forces $rank(\{M \in \mtcl N_{\dot G_\b} \cap \mtcl M^{\b+2}_\ast\,:\,\d_M\notin \dot C_{f(\n)}\}, N)=\n'$ for some given $\n'<\n$ for every (equivalently, for some) $N\in \mtcl O'$ such that $\d_N=f(\n)$.

\end{itemize}

It is enough to assume (b) and show that $q'$ forces that a tail of $range(\dot F_\b\restr\n)$ is included in $\dot C_{f(\n)}$. For this, fix an $N$ as in (b) and, extending $q'$ if necessary, fix also $x\in N$ such that $q'\av_\b$ forces that if $M\in N$ is such that $x\in M$ and $rank(\mtcl N_{\dot G_\b}\cap\mtcl M^{\b+2}_\ast, M)>\n'$, then $\d_M\in \dot C_{f(\n)}$. By further extending $q'$ if necessary we may assume that $x\in M$ for some $M\in \D_{q'}^{-1}(\b)\cap\mtcl M^{\b+2}_\ast\cap N$ such that $x\in M$ and $\d_M=f(\n_\ast)$ for some $\n_\ast\geq\n'$ below $\n$. Now suppose $q''=(p'', \D_{q''})$ extends $q'$ and suppose $\n_\circ\in dom(f^{p'', \b})$ is in $[\n_\ast,\,\n)$. It suffices to show that $q''\av_\b$ forces $f^{p'', \b}(\n_\circ)\in\dot C_{f(\n)}$.

For this, note that $q''\av_\b$ forces that $f^{p'', \b}(\n_\circ)$ is $\d_{M_\circ}$ for some $M_\circ\in \mtcl O^{p'', \b}\sub \mtcl M^{\b+2}_\ast\cap\mtcl N_{\dot G_\b}$ such that $rank(\mtcl N_{\dot G_\b}\cap\mtcl M^{\b+2}_\ast, M_\circ)\geq\n_\circ$. By symmetry of $\mtcl O^{p'', \b}$\footnote{Specifically, by condition $(\g)$ in the definition of symmetric system.} and since $\d_{M^\circ}>\d_M$ there is then some $M'_\circ\in\mtcl O^{p'', \b}$ such that $M\in M'_\circ$ and $\d_{M'_\circ}=\d_{M_\circ}$. Since, by symmetry, $q''\av_\b$ forces $rank(\mtcl N_{\dot G_\b}\cap\mtcl M^{\b+2}_\ast, M'_\circ) = rank(\mtcl N_{\dot G_\b}\cap\mtcl M^{\b+2}_\ast, M_\circ)\geq \n_\circ$ and since $x\in M'_\circ$, it follows that $q''_\b$ forces $f^{p'', \b}(\n_\circ)=\d_{M'_\circ}\in\dot C_{f(\n)}$, which is what we wanted.
\end{proof}

\begin{corollary}
$\mtcl P_\k$ forces measuring.
\end{corollary}

The above corollary follows from Lemmas \ref{measuring}, \ref{cc} and \ref{horribilis} (since $\Phi$ was chosen to be a book-keeping function), and finishes the proof of Theorem \ref{mainthm}.

\end{document}